\documentclass[11pt,reqno]{amsart}
\usepackage[margin=1in]{geometry}
\usepackage{paralist}
\usepackage{tikz}

\theoremstyle{definition}
\newtheorem{definition}{Definition}[section]
\newtheorem{theorem}[definition]{Theorem}
\newtheorem{proposition}[definition]{Proposition}
\newtheorem{lemma}[definition]{Lemma}
\newtheorem{notation}[definition]{Notation}
\newtheorem{observation}[definition]{Observation}
\newtheorem{corollary}[definition]{Corollary}
\newtheorem{example}[definition]{Example}

\DeclareMathOperator{\hd}{HD}
\DeclareMathOperator{\pr}{Pr}

\newcommand{\dkw}{Dvoretzky-Kiefer-Wolfowitz}
\newcommand{\gc}{Glivenko-Cantelli}
\newcommand{\vc}{Vapnik-\v{C}hervonenkis}

\title{Error Probabilities for Halfspace Depth}
\author{Michael A. Burr}
\thanks{This work was partially supported by grants from the Simons Foundation (\#282399 to Michael Burr) and the NSF (\#CCF-1527193)}
\address{Department of Mathematical Sciences, Clemson University, Clemson, SC 29634}
\email{burr2@clemson.edu}
\author{Robert Fabrizio}
\address{Department of Mathematical Sciences, Clemson University, Clemson, SC 29634}
\email{rfabriz@clemson.edu}
\date{\today}


\begin{document}
\begin{abstract}
Data depth functions are a generalization of one-dimensional order statistics and medians to real spaces of dimension greater than one; in particular, a data depth function quantifies the centrality of a point with respect to a data set or a probability distribution.  One of the most commonly studied data depth functions is halfspace depth.  It is of interest to computational geometers because it is highly geometric, and it is of interest to statisticians because it shares many desirable theoretical properties with the one-dimensional median.  As the sample size increases, the halfspace depth for a sample converges to the halfspace depth for the underlying distribution, almost surely.  In this paper, we use the geometry of halfspace depth to improve the explicit bounds on the rate of convergence.
\end{abstract}
\maketitle

\section{Introduction}
Data depth functions generalize order statistics and the median in one dimension to higher dimensions; in particular, they provide a quantitative estimate for the centrality of a point relative to a data set or a probability distribution (see \cite{LiuSingh:Survey:1999} and \cite{Liu:NewSurvey:2003} for surveys).  For these functions, larger values at a point indicate that the point is deeper or more central with respect to a data set or distribution.  The point with largest depth is defined to be the {\em median} of a data set or distribution.  For data sets, data depth functions are typically defined in terms of the geometry of the data, and they reflect the geometric combinatorics of the data set.  

Halfspace depth is a data depth measure that was first introduced by Hodges \cite{Hodges:1955} and Tukey \cite{Tukey:1975}.  Halfspace depth has attracted the interest of computational geometers because of its strong geometric properties (see, for example, \cite{Aloupis:2006} and \cite{Burr:2011}) and is of interest to statisticians because it shares many theoretical properties with the one-dimensional median \cite{ZuoSerfling:2000}.  We recall the definition of halfspace depth for distributions and data sets, briefly using the notation $\mathcal{H}(q)$ for the set of halfspaces in $\mathbb{R}^d$ containing $q\in\mathbb{R}^d$.
\begin{definition}
Let $X$ be an $\mathbb{R}^d$-valued random variable.  For a point $q\in\mathbb{R}^d$, the {\em halfspace depth} of $q$ with respect to $X$ is the following minimum over all halfspaces $H$ of $\mathbb{R}^d$ containing $q$:
$$
\hd(q;X)=\min_{H\in\mathcal{H}(q)}\pr(X\in H).
$$
Let $X^{(n)}=(X_1,\cdots,X_n)$ be a finite sample of $n$ points in $\mathbb{R}^d$.  For a point $q\in\mathbb{R}^d$, the {\em halfspace depth} of $q$ with respect to $X^{(n)}$ is the following minimum over all halfspaces $H$ of $\mathbb{R}^d$ containing $q$:
$$
\hd(q;X^{(n)})=n^{-1}\cdot\min_{H\in\mathcal{H}(q)}\#\{X^{(n)}\cap H\}.
$$
\end{definition}

In \cite{DonohoGasko:1992}, the authors prove that for iid $X_1$, $X_2$, $\ldots$ random variables equal in distribution to $X$, as $n\rightarrow\infty$
$$
\sup_{q\in\mathbb{R}^d}|\hd(q;X)-\hd(q;X^{(n)})|\rightarrow 0\qquad\text{a.s.}
$$
In this paper, we use the geometry of halfspace depth to improve explicit bounds on the rate of convergence of this limit.  In particular, we show:
\begin{theorem}\label{theorem:main:simplified}
Let $X$ be an $\mathbb{R}^d$-valued random variable (obeying certain Lipschitz continuity conditions defined below) and let $(X_1,X_2,\cdots)$ be a sequence of iid random variables, equal in distribution to $X$.  Fix $\varepsilon>0$.  Then, for $n$ sufficiently large, there is a constant $C$ such that 
$$
\pr\left(\sup_{q\in\mathbb{R}^d}|\hd(q;X)-\hd(q;X^{(n)})|\leq\varepsilon\right)\geq1-Cn^{\frac{3}{2}(d-1)}e^{-2n\varepsilon^2}.
$$
\end{theorem}

This theorem represents an improvement by a factor of $n^{\frac{1}{2}d+\frac{7}{2}}$ over the previous bound.  The remainder of this paper is organized as follows: In Section \ref{Section:HalfspaceBackground}, we provide the necessary background on halfspace depth and discuss what is known about its convergence.  In Section \ref{Section:Tools}, we quote results from computational geometry and probability theory that we need in the main results.  In this section, we also provide the notation for the Lipschitz continuity conditions mentioned in Theorem \ref{theorem:main:simplified}.  In Section \ref{Section:Examples}, we illustrate the main theorem with a few examples and special cases, and, finally, we conclude in Section \ref{section:Conclusion}.

\section{Background on Halfspace Depth and its Convergence}\label{Section:HalfspaceBackground}
We begin this section by presenting equivalent formulations of halfspace depth and discuss the geometric properties and notations that are used in the remainder of this paper.  In the second part of this paper, we recall probabilistic estimates on the error of the empirical measure of a sample and apply these estimates to halfspace depth.

\subsection{Equivalent Definitions of Halfspace Depth}\label{Section:Halfspace:Definitions}
Halfspace depth is a commonly studied and used data depth function because it is simple to define, can be computed efficiently, and satisfies all of the desirable properties for a data depth function as defined in \cite{ZuoSerfling:2000}.  In the context of this paper, we recall two equivalent formulations of halfspace depth.  

Observe, first, that a parallel translation of the boundary of a halfspace in the direction of the halfspace only decreases the measure or number of points in the halfspace.  Therefore, in the definition of halfspace depth, it is enough to only consider halfspaces where $q$ lies in the boundary of $H$.  Therefore, we can rewrite the definition of halfspace depth as follows:
\begin{proposition}\label{proposition:boundary:halfspace}
Let $X$ be an $\mathbb{R}^d$-valued random variable and $X^{(n)}=(X_1,\cdots,X_n)$ a finite sample of $n$ points in $\mathbb{R}^d$.  For a point $q\in\mathbb{R}^d$, the halfspace depth of $q$ with respect to $X$ or $X^{(n)}$ is the following minimum over all halfspaces $H$ of $\mathbb{R}^d$ whose boundary $\partial H$ contains $q$:
$$
\hd(q;X)=\min_{\partial H\ni q}\pr(X\in H)\qquad\text{and}\qquad\hd(q;X^{(n)})=n^{-1}\cdot\min_{\partial H\ni q}\#\{X^{(n)}\cap H\}.
$$
\end{proposition}

Next, we reinterpret this definition in terms of projections onto one-dimensional subspaces of $\mathbb{R}^d$ as in \cite{DonohoGasko:1992}; this allows us to interpret a high-dimensional problem as a collection of one-dimensional problems.  For $q\in\mathbb{R}^d$, the set of halfspaces $H$ with $q\in\partial H$ can be parametrized by points in the $(n-1)$-dimensional sphere, $S^{n-1}$.  More precisely, each direction in $S^{n-1}$ corresponds to a vector $v$ which describes a halfspace $H$ as follows: the bounding hyperplane $\partial H$ of the halfspace passes through $q$ and has normal $v$ and the halfspace opens in the direction away from $v$.

We now take this idea and look at it from a different perspective; instead of fixing the point $q$ and considering all halfspaces whose boundary passes through $q$, we focus on the parameterization given by $S^{d-1}$.  In order to simplify the discussion, we use the following notation:
\begin{notation}
Let $\theta\in S^{d-1}$, define $u_\theta$ to be the vector in $\mathbb{R}^d$ pointing in the direction of $\theta$, $\ell_\theta$ to be the line through the origin in the direction of $\theta$, and $\pi_\theta:\mathbb{R}^d\rightarrow\ell_\theta$ to be the orthogonal projection onto $\ell_\theta$.  Moreover, for a point $p\in \mathbb{R}^d$, define $d_\theta(p)=\langle p,u_\theta\rangle$ (where $\langle\cdot,\cdot\rangle$ is the standard inner product in $\mathbb{R}^d$) to be the signed length of $p$ in the direction of $u_\theta$.  In other words, $\pi_\theta(p)=d_\theta(p)u_\theta$.
\end{notation}

\noindent Using this notation, we define probability distributions and finite samples on $\mathbb{R}$ corresponding to each direction in $S^{d-1}$.

\begin{notation}
Let $X$ be an $\mathbb{R}^d$-valued random variable and $X^{(n)}=(X_1,\cdots,X_n)$ a finite sample of $n$ points in $\mathbb{R}^d$.  For each $\theta\in S^{n-1}$, $X_\theta$ is the $\mathbb{R}$-valued random variable $d_\theta(X)$, and $F_\theta$ is the cdf for this variable, i.e., $F_\theta(t)=\pr(X_\theta\leq t)$.  Similarly, $F_{n,\theta}$ is the empirical cdf for the points of $X^{(n)}$ in the direction of $\theta$, i.e., $F_{n,\theta}(t)=n^{-1}\cdot\#\{i:d_\theta(X_i)\leq t\}$.
\end{notation}

For each $t$ and $\theta$, there is a halfspace $H_{\theta,t}$ such that $F_\theta(t)=\pr(X\in H_{\theta,t})$ and $F_{n,\theta}(t)=n^{-1}\cdot\#\{X^{(n)}\cap H_{\theta,t}\}$.  In particular, $H_{\theta,t}$ is the halfspace whose bounding hyperplane passes through the point $t\cdot u_\theta$, whose bounding hyperplane has normal $u_\theta$, and the halfspace opens in the direction of $-u_\theta$.  Since the point $q$ lies in the bounding hyperplane for $H_{\theta,t}$ iff $d_\theta(q)=t$, we can reinterpret Proposition \ref{proposition:boundary:halfspace} as follows (see Figure \ref{figure:projections}):

\begin{proposition}\label{proposition:directions}
Let $X$ be an $\mathbb{R}^d$-valued random variable and $X^{(n)}=(X_1,\cdots,X_n)$ a finite sample of $n$ points in $\mathbb{R}^d$.  For a point $q\in\mathbb{R}^d$, the halfspace depth of $q$ with respect to $X$ or $X^{(n)}$ is the following minimum over directions $\theta\in S^{d-1}$:
$$
\hd(q;X)=\min_{\theta\in S^{d-1}}F_\theta(d_\theta(q))\qquad\text{and}\qquad\hd(q;X^{(n)})=\min_{\theta\in S^{d-1}}F_{n,\theta}(d_\theta(q)).
$$
\end{proposition}

\begin{figure}[hbt]
\begin{center}
\begin{tikzpicture}[scale=1.3]
\draw[color = lightgray](-1,0) -- (3,0);
\draw[color = lightgray](0,-1) -- (0,3);
\draw (0,.4)  node (C){};
\begin{scope}
\clip (-1,-1) -- (-1,3) -- (3.2,3) -- (3.2,-1) -- cycle;
\draw (0,2) node (B) {} ;
\begin{scope}[rotate=30]
\draw (-2,0) -- (3,0) node (A) {};
\draw (1.5,2.7) -- (1.5,-2);
\foreach \x in {0,0.2,...,6}
\draw (1.6,\x-3) -- (1.9,\x-2.7);
\draw[->,line width=2pt] (0,0) -- (1,0);
\filldraw[color = gray] (1.5,1.5) circle (.15);
\draw (1.35,1.5) node(D){};
\filldraw (1.5,0) circle (.15);
\draw (1.45,.2) node(E){};
\filldraw[color = gray] (-.3,1) circle (.15);
\draw (-.3,.85) node (F){};
\filldraw (-.3,0) circle (.15);
\draw (-.3,.15) node (G){};
\filldraw[color = gray] (2.3,1.8) circle (.15);
\draw (2.3,1.65) node (H){};
\filldraw (2.3,0) circle (.15);
\draw (2.3,.15) node (I){};
\filldraw[color = gray] (2.7,-1) circle (.15);
\draw (2.7,-.85) node (J){};
\filldraw (2.7,0) circle (.15);
\draw (2.7,-.15) node (K){};
\end{scope}
\end{scope}
\draw (A) node[right]{$\ell_\theta$};
\draw (B) node[left] {$H_{\theta,d_\theta(q)}$};
\draw (C) node[right] {$u_\theta$};
\draw (D) node[left] {$q$};
\draw (E) node[left]{$\pi_\theta(q)$};
\draw[dashed,->] (F)--(G);
\draw[dashed,->](H)--(I);
\draw[dashed,->](J)--(K);
\end{tikzpicture}
\end{center}
\caption{The number of points in the halfspace $H_{\theta,d_{\theta}(q)}$ (the unshaded halfspace) equals the number of points such that $d_\theta(X_i)\leq d_\theta(q)$.  In the diagram, these are the points such that $\pi_\theta(X_i)$ is to the left of $\pi_\theta(q)$.  A similar statement and diagram can be made for a $\mathbb{R}^d$-valued random variable $X$.\label{figure:projections}}
\end{figure}
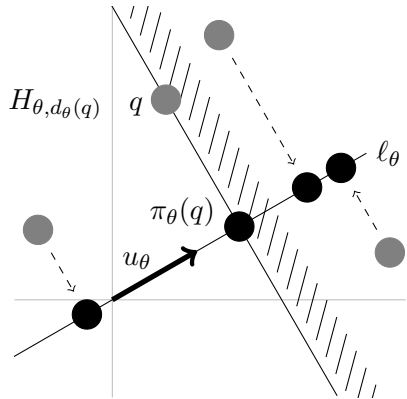

These equivalent formulations for halfspace depth illustrate that the halfspace depth of a point $q$ measures how extreme $q$ is under all orthogonal projections onto one-dimensional subspaces.  This formulation gives rise to the common description of the halfspace depth of a point $q$ relative to a sample $X^{(n)}$ as the smallest fraction of points that must be removed from the sample so that $q$ is outside the convex hull of the remaining sample points.  The formulation of halfspace depth in Proposition \ref{proposition:directions} gives the key approach that is exploited in the main result of this paper.

\subsection{Prior Convergence Estimates}

Suppose that $X$ is $\mathbb{R}^d$-valued random variable and that $(X_1,X_2,\cdots)$ is a sequence of iid random variables equal in distribution to $X$.  In \cite{DonohoGasko:1992}, the authors prove that for all $q\in\mathbb{R}^d$, $\hd(q;X^{(n)})\rightarrow\hd(q;X)$ almost surely as $n\rightarrow\infty$.  This result is proved by observing that the collection of all halfspaces in $\mathbb{R}^d$ satisfy the \gc\ property (for additional details, see \cite{ShorackWellner:Book:1986}), so that, uniformly for all halfspaces $H$, $n^{-1}\cdot\#\{X^{(n)}\cap H\}\rightarrow\pr(X\in H)$ a.s. as $n\rightarrow\infty$.

The convergence given by the \gc\ property can be strengthened by observing that the set of halfspaces are a \vc\ class.  In particular, it is shown in \cite{Dudley:1979} that halfspaces in $\mathbb{R}^d$ cannot shatter sets of size $d+2$ (and convex sets of size $d+1$ can be shattered).  We write $m(n)$ for the maximum number of subsets formed by intersecting finite samples of size $n$ with halfspaces in $\mathbb{R}^d$.  Then, in \cite{VapnikCervonenkis:1971} and \cite{Devroye:1982}, the following result is proved:
\begin{proposition}[See {\cite[Chapter 26]{ShorackWellner:Book:1986}}]\label{proposition:VC}
Let $\mathcal{H}$ be the set of all halfspaces in $\mathbb{R}^d$ and suppose that $\varepsilon>0$.  Define $m$ to be the function as described above.  Then, for $n$ sufficiently large, 
\begin{enumerate}
\item $\displaystyle\pr(\sup_{H\in\mathcal{H}}|n^{-1}\cdot\#\{X^{(n)}\cap H\}-\pr(X\in H)|\geq\varepsilon)\leq 4m(2n)e^{-n\varepsilon^2/8}$.\label{inequality:1}
\item $\displaystyle\pr(\sup_{H\in\mathcal{H}}|n^{-1}\cdot\#\{X^{(n)}\cap H\}-\pr(X\in H)|\geq\varepsilon)\leq 4m(n^2)e^{-2n\varepsilon^2}$.\label{inequality:2}
\end{enumerate}
Moreover, $m(r)\leq\frac{3}{2}\cdot\frac{r^{d+1}}{(d+1)!}$.
\end{proposition}

We improve these bounds on the error by decreasing the exponent or the degree of the polynomial coefficient.

\section{Additional Tools}\label{Section:Tools}

Since data depth combines discrete geometry with probability and statistics, our main result requires results from both of these fields.  In this section, for convenience, we collect a few additional theorems and notations that are used in the remainder of the paper.

\subsection{Spherical Covering}

As observed in Section \ref{Section:Halfspace:Definitions}, the halfspace depth is related to the directions in a $(d-1)$-dimensional sphere.  In our main result, we consider small neighborhoods on a $(d-1)$-dimensional sphere.  The following result indicates how many of these neighborhoods are necessary.

\begin{lemma}[{\cite[Corollary 1.2]{Wintsche:2003}}]\label{lemma:covering}
For any $0<\psi<\arccos(d^{-1/2})$, there is an absolute constant $C_2$ such that the $(d-1)$-dimensional unit sphere can be covered by 
$$
C_2\frac{\cos\psi}{\sin^{d-1}\psi} (d-1)^{\frac{3}{2}}\ln\left(1+(d-1)\cos^2\psi\right)\leq C_2\left(\frac{\sqrt{d}}{\psi}\right)^{d-1}(d-1)^{\frac{3}{2}}\ln(d)
$$
spherical balls of radius $\psi$ (i.e., the radius on the surface of the sphere is $\psi$).
\end{lemma}

The final inequality in this lemma follows from the facts that $\cos(\psi)\leq 1$ and for $0<\psi<\arccos(d^{-1/2})$, $\sin(\psi)\geq\psi\cdot\cos(\psi)>d^{-1/2}\psi$.  The $\sqrt{d}$ could be replaced by a constant by using a better bound on $\sin(x)$, such as $\sin(x)\geq x-\frac{x^3}{3!}$.  We leave the details to the interested reader.

\subsection{The \dkw\ Ineqaulity}

In the one-dimensional case, the bounds in Proposition \ref{proposition:VC} can be improved with the \dkw\ inequality.
\begin{lemma}[See, for example, \cite{DevroyeLugosi:2001}]\label{lemma:DKW}
Let $X$ be an $\mathbb{R}$-valued random variable, and let $(X_1,X_2,\cdots)$ be a sequence of iid random variables equal in distribution to $X$.  Let $F$ be the cdf of $X$, and let $F_n$ be the empirical cdf, i.e., $F_n(t)=n^{-1}\cdot\#\{i\leq n:X_i\leq t\}$.  For each $\varepsilon>0$,
$$
\pr(\sup_{t\in\mathbb{R}}|F(t)-F_n(t)|\geq\varepsilon)\leq 2e^{-2n\varepsilon^2}.
$$
In this paper, we use the \dkw\ inequality on one-dimensional projections, as in Proposition \ref{proposition:directions}, in order to improve the bound on the convergence rate of the sample halfspace depth.
\end{lemma}

\subsection{Lipschitz and Tail Behavior}

In this section, we define a few conditions on the probability distributions that we use for our main result.  Let $X$ be an $\mathbb{R}^d$-valued random variable with $d>1$.

\begin{definition}\label{definition:decaysquickly}
We say that an $\mathbb{R}^d$-valued random variable $X$ {\em decays quickly} if there is some $\lambda>0$ such that $\pr(|X|>R)=O\left(R^{3d-5}\cdot e^{-\lambda \frac{R^2}{2}}\right)$.  In other words, there exists a constant $C_1>0$ such that for $R>1$, $\pr(|X|>R)<C_1\cdot R^{3d-5}\cdot e^{-\lambda \frac{R^2}{2}}$.  We say that $\lambda$ is the {\em decay rate} of $X$.
\end{definition}

For example, the multivariate normal with the origin the mean and covariance matrix $I$ decays exponentially with $\lambda=1$.

\begin{definition}
We say that $X$ is {\em Lipschitz continuous in projection} if there is a constant $L_{\pi}>0$ such that for all $\theta\in S^{d-1}$, $F_{\theta}$ is Lipschitz continuous with Lipschitz constant $L_{\pi}$.  We say that $X$ is {\em radially Lipschitz continuous} if there is a constant $L_\theta$ such that for any fixed $t$, the function $F_\theta(t)$ is Lipschitz continuous as a function of $\theta$ on the unit circle with Lipschitz constant $L_\theta$.  We say that $L_\pi$ is the {\em projection Lipschitz constant} and $L_\theta$ is the {\em radial Lipschitz constant} for $X$.
\end{definition}

For example, radially symmetric distributions about the origin have $L_\theta=0$.  These two Lipschitz constants indicate that as $\theta$ and $t$ change, $F_\theta(t)$ changes continuously.  In other words, these two constants can be used to bound the difference in the probability between two halfspaces and to show that this difference varies continuously.

\section{Main Result}\label{section:MainResult}
We begin this section by highlighting the difference between the previous bounds and our approach.  In Proposition \ref{proposition:VC}, the convergence rates are computed over all halfspaces in $\mathbb{R}^d$, and the goal is to find a bound on 
$$
\pr(\sup_{H\in\mathcal{H}}|n^{-1}\cdot\#\{X^{(n)}\cap H\}-\pr(X\in H)|\geq\varepsilon).
$$
In our approach, however, the goal is to find bounds on the family of one-dimensional cdfs $F_\theta$ and $F_{n,\theta}$.  In particular, in our main result, we find a uniform upper bound on the following probability:
$$
\pr(\sup_{\substack{t\in\mathbb{R}\\\theta\in S^{d-1}}}|F_\theta(t)-F_{n,\theta}(t)|\geq\varepsilon).
$$
Then, by applying Proposition \ref{proposition:directions}, when $|F_\theta(t)-F_{n,\theta}(t)|<\varepsilon$ for all $t$ and $\theta$, it follows that for all $q\in\mathbb{R}^d$, $|\hd(q;X)-\hd(q;X^{(n)})|<\varepsilon$.  The advantage of this approach is that by considering only one-dimensional objects, we can use the improved bounds of the \dkw\ inequality.

Throughout the remainder of this section, we assume that the $\mathbb{R}^d$-valued random variable $X$ with $d>1$ has decay rate $\lambda$ with constant $C_1$, projection Lipschitz constant $L_\pi$, and radial Lipschitz constant $L_\theta$.  Also, let $(X_1,X_2,\cdots)$ be a sequence of iid random variables equal in distribution to $X$.

Our proof follows the following steps: \begin{inparaenum}
\item First, we quantify how small changes in $\theta$ affect the cdfs $F_\theta$ and $F_{n,\theta}$.
\item Second, we use this description on $F_\theta$ and $F_{n,\theta}$ to argue that it is sufficient to examine finitely many $\theta$'s instead of the infinitely many possible $\theta$'s in $S^{d-1}$.
\item Finally, we carefully choose a few parameters in order to achieve a bound which improves upon the bounds achieved in Proposition \ref{proposition:VC}.
\end{inparaenum}

For our first step, we begin by collecting, for later convenience, a consequence of the Lipschitz condition, and then provide a geometric description of the behavior of $d_\theta$ as $\theta$ changes.  For notational convenience, for $\theta,\varphi\in S^{d-1}$, we write $|\theta-\varphi|$ for the spherical distance between them, i.e., the central angle between the two points.
\begin{observation}\label{observation:Lipschitz}
For all $\theta,\varphi\in S^{d-1}$ and $t\in\mathbb{R}$, $|F_\theta(t)-F_\varphi(t)|\leq L_\theta|\theta-\varphi|$.
\end{observation}

\begin{lemma}\label{lemma:geometric:change}
Let  $\theta,\varphi\in S^{d-1}$ and $x\in\mathbb{R}^d$.  Then, $|d_\theta(x)-d_\varphi(x)|\leq\|x\|\cdot|\theta-\varphi|$.
\end{lemma}
\begin{proof}
Observe first that $|d_\theta(x)-d_\varphi(x)|=|\langle x,u_\theta-u_\varphi\rangle|\leq\|x\|\cdot\|u_\theta-u_\varphi\|$.  Note that the vector $u_\theta-u_\varphi$ is the chord of an arc on the great circle between $\theta$ and $\varphi$; since chords are shorter than the arcs they cut off, and the length of the arc is $|\theta-\varphi|$, the result follows.
\end{proof}

\noindent Moreover, Lemma \ref{lemma:geometric:change} allows us to bound $F_{n,\varphi}$ in terms of $F_{n,\theta}$.  More precisely:
\begin{corollary}\label{corollary:inequalities}
Fix $R>0$ and $\theta,\varphi\in S^{d-1}$.  Suppose that for all $i\leq n$, $\|X_i\|\leq R$.  Then,
$$
F_{n,\theta}(t-R|\theta-\varphi|)\leq F_{n,\varphi}(t)\leq F_{n,\theta}(t+R|\theta-\varphi|).
$$
\end{corollary}
\begin{proof}
This result follows from Lemma \ref{lemma:geometric:change}, because we know that if $d_{\varphi}(X_i)\leq t$, then $d_{\theta}(X_i)\leq t+R\cdot|\theta-\varphi|$, and if $d_{\theta}(X_i)\leq t-R\cdot|\theta-\varphi|$, then $d_{\varphi}(X_i)\leq t$.  
\end{proof}

With Observation \ref{observation:Lipschitz} and Corollary \ref{corollary:inequalities} in hand, we know how the distributions $F_\theta$ and $F_{n,\theta}$ vary as $\theta$ varies.  This leads to the following proposition, which bounds the error between $F_\theta$ and $F_{n,\theta}$ as $\theta$ varies.

\begin{proposition}
Fix $R>0$, $\varepsilon>0$, $\delta>0$, and $\theta,\varphi\in S^{d-1}$.  Suppose that for all $i\leq n$, $\|X_i\|\leq R$.  Suppose that $\sup_{t\in\mathbb{R}}|F_\theta(t)-F_{n,\theta}(t)|<\varepsilon\cdot (1+\delta)^{-1}$.  Then, $$\sup_{t\in\mathbb{R}}|F_{\varphi}(t)-F_{n,\varphi}(t)|<\varepsilon(1+\delta)^{-1}+L_\theta|\theta-\varphi|+L_\pi R|\theta-\varphi|.$$
\end{proposition}
\begin{proof}
By Corollary \ref{corollary:inequalities}, we know that 
$$
F_{n,\theta}(t-R|\theta-\varphi|)-F_{\varphi}(t)\leq F_{n,\varphi}(t)-F_{\varphi}(t)\leq F_{n,\theta}(t+R|\theta-\varphi|)-F_{\varphi}(t).
$$
Then, it follows that
$$
|F_{n,\varphi}(t)-F_{\varphi}(t)|\leq\max\{|F_{n,\theta}(t-R|\theta-\varphi|)-F_{\varphi}(t)|,|F_{n,\theta}(t+R|\theta-\varphi|)-F_{\varphi}(t)|\}.
$$
We proceed to bound the first of the expression in the maximum (the other expression is similar).  By the triangle inequality, 
\begin{align*}
|F_{n,\theta}&(t-R|\theta-\varphi|)-F_\varphi(t)|\\&\leq |F_{n,\theta}(t-R|\theta-\varphi|)-F_\theta(t-R|\theta-\varphi|)|+|F_\theta(t-R|\theta-\varphi|)-F_\theta(t)|+|F_\theta(t)-F_{\varphi}(t)|.\end{align*}
By assumption, the first expression is bounded above by $\varepsilon\cdot(1+\delta)^{-1}$.  By the projection Lipschitz constant, the second expression is bounded above by $L_\pi\cdot R\cdot |\theta-\varphi|$.  Finally, by Observation \ref{observation:Lipschitz}, the third expression is bounded by $L_\theta\cdot|\theta-\varphi|$.  Combining these bounds, the result follows.
\end{proof}

\noindent By choosing $|\theta-\varphi|$ to be sufficiently small, as in the following corollary, we can insist that the errors are all bounded above by $\varepsilon$. 

\begin{corollary}\label{corollary:neighborhoods}
Fix $R>0$, $\varepsilon>0$, $\delta>0$, and $\theta\in S^{d-1}$.  Suppose that for all $i\leq n$, $\|X_i\|\leq R$.  Let $\varphi\in S^{d-1}$ be such that $|\theta-\varphi|<\varepsilon\cdot\delta\cdot(1+\delta)^{-1}\cdot(L_\theta+L_\pi\cdot R)^{-1}$.  Suppose that $\sup_{t\in\mathbb{R}}|F_\theta(t)-F_{n,\theta}(t)|<\varepsilon\cdot (1+\delta)^{-1}$.  Then, $\sup_{t\in\mathbb{R}}|F_{\varphi}(t)-F_{n,\varphi}(t)|<\varepsilon.$
\end{corollary}

Corollary \ref{corollary:neighborhoods} completes the second main step of the proof and gives us a way to study neighborhoods on $S^{d-1}$ instead of individual projections.  In particular, if $\{\theta_1,\cdots,\theta_k\}$ satisfy the conditions of Corollary \ref{corollary:neighborhoods}, then we have estimates on the error in the $k$ spherical balls of radius $\varepsilon\cdot\delta\cdot(1+\delta)^{-1}\cdot(L_\theta+L_\pi\cdot R)^{-1}$ centered at the $\theta_i$'s.  We now determine the $\theta_i$'s and use the \dkw\ bound in these directions to get our initial estimate on the error between the halfspace depth of a sample and the underlying distribution.  In particular,

\begin{proposition}\label{proposition:firstprobability}
Fix $R>1$, $\varepsilon>0$, and $\delta>0$.  Suppose that $\varepsilon\cdot\delta\cdot(1+\delta)^{-1}\cdot(L_\theta+L_\pi\cdot R)^{-1}<\arccos(d^{-1/2})$.  Then,
\begin{multline}\label{equation:Firstbound}
\pr\left(\sup_{q\in\mathbb{R}^d}|\hd(q;X)-\hd(q;X^{(n)})|\leq\varepsilon\right)\geq\\1-2C_2\left(\frac{(1+\delta)(L_\theta+L_\pi R)\sqrt{d}}{\varepsilon\delta}\right)^{d-1} (d-1)^{\frac{3}{2}}\ln(d)e^{-2n\varepsilon^2(1+\delta)^{-2}}-C_1 n R^{3d-5} e^{-\lambda\frac{R^2}{2}}.
\end{multline}
\end{proposition}
\begin{proof}
Observe first, that, in each of the statements above, we require that for all $i\leq n$, $\|X_i\|\leq R$.  By the exponential decay property, for each $i$, the probability of this occurring is bounded below by $1-C_1R^{3d-5}e^{-\lambda\frac{R^2}{2}}$.

Next, observe that if $\theta\in S^{d-1}$ is such that $\sup_{t\in\mathbb{R}}|F_\theta(t)-F_{n,\theta}(t)|<\varepsilon\cdot (1+\delta)^{-1}$, then for all $\varphi\in S^{d-1}$ within a spherical ball of radius $\varepsilon\cdot\delta\cdot(1+\delta)^{-1}\cdot(L_\theta+L_\pi\cdot R)^{-1}$ satisfy the conditions of Corollary \ref{corollary:neighborhoods}.  Therefore, we cover the $(d-1)$-dimensional sphere with balls of radius $\varepsilon\cdot\delta\cdot(1+\delta)^{-1}\cdot(L_\theta+L_\pi\cdot R)^{-1}$.  By Lemma \ref{lemma:covering}, $C_2\cdot\left((1+\delta)(L_\theta+L_\pi R)\sqrt{d}\right)^{d-1}\cdot(\varepsilon\delta)^{1-d}\cdot (d-1)^{\frac{3}{2}}\cdot\ln(d)$ such spherical balls are required.  Fix $\theta$ to be the center of one of these spherical balls.  Applying the \dkw\ inequality, Lemma \ref{lemma:DKW}, it follows that
$$
\pr(\sup_{t\in\mathbb{R}}|F_\theta(t)-F_{n,\theta}(t)|\leq\varepsilon (1+\delta)^{-1})\geq 1-2e^{-2n\varepsilon^2(1+\delta)^{-2}}.
$$
Since the spherical balls cover $S^{d-1}$, for any $\varphi\in S^{d-1}$, there is some spherical ball with center $\theta$ such that $|\theta-\varphi|<\varepsilon\cdot\delta\cdot(1+\delta)^{-1}\cdot(L_\theta+L_\pi\cdot R)^{-1}$.  Hence, by Corollary \ref{corollary:neighborhoods}, it follows that $\sup_{t\in\mathbb{R}}|F_{\varphi}(t)-F_{n,\varphi}(t)|<\varepsilon$.

Finally, using the fact that for $a,b>0$, $(1-a)(1-b)>1-(a+b)$, and that there are $n$ sample points and $C_2\cdot\left((1+\delta)(L_\theta+L_\pi R)\sqrt{d}\right)^{d-1}\cdot(\varepsilon\delta)^{1-d}\cdot (d-1)^{\frac{3}{2}}\cdot\ln(d)$ spherical balls, the result follows.
\end{proof}

We now simplify the expression above by eliminating the parameters $R$ and $\delta$.  Since the bound in Inequality (\ref{equation:Firstbound}) has two exponentials, we can choose $R=\varepsilon\cdot\frac{2\sqrt{n}}{\sqrt{\lambda}(1+\delta)}$ to equate the exponentials.  This choice results in the following corollary:

\begin{corollary}\label{corollary:secondstep}
Fix $\varepsilon>0$ and $\delta>0$.  Suppose that $\varepsilon\cdot\delta\cdot(1+\delta)^{-1}\cdot\left(L_\theta+L_\pi\cdot \varepsilon\cdot\frac{2\sqrt{n}}{\sqrt{\lambda}(1+\delta)}\right)^{-1}<\arccos(d^{-1/2})$ and $2\varepsilon\sqrt{n}>\sqrt{\lambda}(1+\delta)$.  Then,
\begin{multline}\label{equation:Secondbound}
\pr\left(\sup_{q\in\mathbb{R}^d}|\hd(q;X)-\hd(q;X^{(n)})|\leq\varepsilon\right)\geq\\1-\left(2C_2\left(\frac{(L_\theta\sqrt{\lambda}(1+\delta)+2 L_\pi\sqrt{n}\varepsilon)\sqrt{d}}{\varepsilon\delta\sqrt{\lambda}}\right)^{d-1} (d-1)^{\frac{3}{2}}\ln(d)+C_1
\left(\frac{\varepsilon}{\sqrt{\lambda}(1+\delta)}\right)^{3d-5}n^{\frac{3}{2}(d-1)}\right)e^{-2n\varepsilon^2(1+\delta)^{-2}}.
\end{multline}
\end{corollary}

\noindent Finally, we choose $\delta=n^{-1}$ in order to eliminate $\delta$.  Additionally, this choice causes the exponential to simplify to $e^{-2n\varepsilon^2}$.  More precisely,

\begin{theorem}\label{Corollary:Final}
Fix $\varepsilon>0$ and suppose that $\varepsilon\cdot (n+1)^{-1}\cdot\left(L_\theta+L_\pi\cdot \varepsilon\cdot\frac{2n^{3/2}}{\sqrt{\lambda}(n+1)}\right)^{-1}<\arccos(d^{-1/2})$ and $2\varepsilon n^{\frac{3}{2}}>\sqrt{\lambda}(n+1)$.  Then,
\begin{multline}
\pr\left(\sup_{q\in\mathbb{R}^d}|\hd(q;X)-\hd(q;X^{(n)})|\leq\varepsilon\right)\geq\\1-\left(2C_2\left(\frac{(L_\theta\sqrt{\lambda}(n+1)+2 L_\pi n^{\frac{3}{2}}\varepsilon)\sqrt{d}}{\varepsilon\sqrt{\lambda}}\right)^{d-1} (d-1)^{\frac{3}{2}}\ln(d)+C_1 \left(\frac{\varepsilon n}{\sqrt{\lambda}(n+1)}\right)^{3d-5}n^{\frac{3}{2}(d-1)} \right)e^4e^{-2n\varepsilon^2}.
\end{multline}
\end{theorem}

\begin{proof}
By setting $\delta=n^{-1}$ in Inequality (\ref{equation:Secondbound}) and expanding the exponential, the bound in Corollary \ref{corollary:secondstep} becomes 
\begin{multline*}
\pr\left(\sup_{q\in\mathbb{R}^d}|\hd(q;X)-\hd(q;X^{(n)})|\leq\varepsilon\right)\geq\\1-\left(2C_2\left(\frac{(L_\theta\sqrt{\lambda}(n+1)+2 L_\pi n^{\frac{3}{2}}\varepsilon)\sqrt{d}}{\varepsilon\sqrt{\lambda}}\right)^{d-1} (d-1)^{\frac{3}{2}}\ln(d)+C_1\left(\frac{\varepsilon n}{\sqrt{\lambda}(n+1)}\right)^{3d-5}n^{\frac{3}{2}(d-1)}\right)e^{-2\varepsilon^2\left(n-\frac{2n^2+n}{n^2+2n+1}\right)}.
\end{multline*}
Since $\varepsilon$ is a difference between two quantities whose values are between $0$ and $1$, only $\varepsilon\leq 1$ has content.  Moreover, the quotient in the exponential is decreases to $-2$ for $n$ positive; therefore, the exponent is bounded above by $-2n\varepsilon^2+4$, and the result follows.
\end{proof}

We summarize Theorem \ref{Corollary:Final} for $n$ sufficiently large.

\begin{corollary}\label{corollary:simplified}
Fix $\varepsilon>0$.  Then, for $n$ sufficiently large, there is a constant $C$ depending on $\lambda$, $\varepsilon$, $C_1$, $C_2$, $d$, and $L_\pi$ such that 
$$
\pr\left(\sup_{q\in\mathbb{R}^d}|\hd(q;X)-\hd(q;X^{(n)})|\leq\varepsilon\right)\geq1-Cn^{\frac{3}{2}(d-1)}e^{-2n\varepsilon^2}.
$$
\end{corollary}

We can compare this result with the known convergence rates in Proposition \ref{proposition:VC}.  The exponential in Inequality (\ref{inequality:1}) is much larger (the exponent is a negative number with smaller magnitude) than what appears in Corollary \ref{corollary:simplified}, so our bound is tighter.  On the other hand, the best coefficient in Inequality (\ref{inequality:2}) is $n^{2d+2}$, which is a factor of $n^{\frac{1}{2}d+\frac{7}{2}}$ times larger than our bound.

\section{Examples}\label{Section:Examples}

In this section, we apply Theorem \ref{Corollary:Final} to several examples.  We take special care to develop explicit bounds on the Lipschitz constants whenever possible.

\begin{example}\label{example:ellipticallysymmetric}
Consider the case where $X$ is distributed according to an elliptically symmetric distribution with pdf $f(x)=\det(\Sigma)^{-\frac{1}{2}}\psi((x-\mu)^T\Sigma^{-1}(x-\mu))$ where $x$ and $\mu$ are in $\mathbb{R}^d$, $\Sigma$ is a positive definite symmetric matrix, and $\psi:[0,\infty)\rightarrow[0,\infty)$ (see, for example, \cite{Chmielewski:1981}).  Since $\Sigma$ is positive definite, we can decompose $\Sigma=QD^2Q^T$ where the entries of $Q$ are the orthogonal eigenvectors of $\Sigma$ and $D$ is a diagonal matrix whose entries are the square roots of the eigenvalues of $\Sigma$.

Consider the affine transformation $Y=D^{-1}Q^T(X-\mu)$; under this transformation, the pdf of $Y$ becomes $\psi(\sum y_i^2)$.  Recall that since halfspace depth is invariant under affine transformations (see, for example, \cite{ZuoSerfling:2000}).  Therefore,  $$\sup_{q\in\mathbb{R}^d}|\hd(q;X)-\hd(q;X^{(n)})|=\sup_{q\in\mathbb{R}^d}|\hd(q;Y)-\hd(q;Y^{(n)})|,$$ and we can study the convergence of $X$ by studying the convergence of $Y$.  Throughout the remainder of this example all statements refer to $Y$.  

We suppose, additionally, that $\psi$ has the following properties:
\begin{enumerate}
\item $\int_0^\infty r^{d-1}\psi(r^2)dr=(\operatorname{Vol}_{d-1}(S^{d-1}))^{-1}$.  This condition guarantees that the integral of the pdf for $Y$ over $\mathbb{R}^d$ is $1$.\label{Condition:Probability}
\item There exists a $\lambda>0$ such that $\int_R^\infty r^{d-1}\psi(r^2)dr=O(R^{3d-5}e^{-\lambda\frac{R^2}{2}})$.  This condition guarantees that $X$ has a decay rate of $\lambda$.\label{Condition:Decay}
\item The function $\Psi(t)=\int_{\mathbb{R}^{d-1}}\psi(t^2+x_1^2+\cdots+x_{d-1}^2)dx_1\cdots dx_{d-1}$ is bounded.  This function is the derivative of $F_{e_d}(t)$, where $e_d$ is the north pole of $S^{d-1}$.  The upper bound on this function is a Lipschitz constant for $F_{e_d}(t)$.  When Conditions (\ref{Condition:Probability}) and (\ref{Condition:Decay}) hold, this is not a strong condition; for example, it holds when $\psi$ is bounded.\label{Condition:Lipschitz}
\end{enumerate}

Like all distributions which are spherically symmetric about the origin, for $Y$, $L_\theta=0$.  Moreover, by Condition (\ref{Condition:Decay}), $Y$ has a decay rate of $\lambda$.  Finally, since this is a spherically symmetric distribution, $F_\theta(t)$ is independent of $\theta$.  By Condition \ref{Condition:Lipschitz}, $F'_{e_d}(t)=\int_{\mathbb{R}^{d-1}}\psi(t^2+x_1^2+\cdots+x_{d-1}^2)dx_1\cdots dx_{d-1}$ is bounded.  Therefore, the projection Lipschitz constant is $L_\pi=\sup_{t\in\mathbb{R}}\left|\int_{\mathbb{R}^{d-1}}\psi(t^2+x_1^2+\cdots+x_{d-1}^2)dx_1\cdots dx_{d-1}\right|$.
\end{example}

In particular cases, we can derive more precise bounds and constants.

\begin{example}
Consider the case where $X$ is distributed according to a non-degenerate normal distribution in $\mathbb{R}^d$ with mean $\mu$ and (positive definite) covariance matrix $\Sigma$.  By applying Example \ref{example:ellipticallysymmetric}, it is enough to consider the standard normal distribution in $\mathbb{R}^d$ centered at the origin with covariance matrix $I$.  In this case, the pdf of $X$ is $(2\pi)^{-d/2}e^{-\frac{1}{2}\sum x_i^2}$.  Moreover, via spherical integration it follows that $\pr(|X|>R)=O(R^{d-2}e^{-\frac{R^2}{2}})$, so the decay rate of $X$ is $1$.  Additionally, for any $\theta$, $F_\theta$ is the cdf of a standard normal distribution in one variable, since the Lipschitz constant for $F_\theta$ is bounded by the absolute value of the derivative, $L_\pi=\frac{1}{\sqrt{2\pi}}$.  

The computation for $C_1$ in Definition \ref{definition:decaysquickly} is more technical (not particularly interesting); we observe that since $\pr(|X|>R)=O(R^{d-2}e^{-\frac{R^2}{2}})$, we could, in fact, replace the $3d-5$ by $d-2$ in Definition \ref{definition:decaysquickly} and subsequent computations to achieve a better bound with $C_1$ appearing only in a lower order term.  Therefore, we choose to ignore the $C_1$ term.  We leave the details to the interested reader.
\end{example}

\begin{example}
For the two-dimensional normal, we can say even more.  More precisely, assume that $X$ is distributed according to a bivariate normal centered at the origin with covariance $I$.  The bounds above apply, but can be made even sharper.  For example, we can cover a circle using at most $\frac{\pi}{w}+1$ intervals of width $2w$.  Moreover, $\pr(|X|>R)=e^{-\frac{1}{2}R^2}$.  Therefore, in this case, the bound in Theorem \ref{Corollary:Final} has a smaller constant, and, explicitly, the bound becomes
$$
1-\left(2\sqrt{2\pi}n^{\frac{3}{2}}+n+2\right)e^4e^{-2n\varepsilon^2}.
$$

Additionally, in the two-dimensional case, the function $m$ for Proposition \ref{proposition:VC} can be computed explicitly.  The largest number of subsets of a set of size $n$ occurs when the points of $n$ are in convex position.  In this case, there are $n^2-n+2$ subsets formed from intersections with halfspaces, so $m(r)=r^2-r+2$.  Even with this smaller degree polynomial for $m$, our bounds are still an improvement by a factor of $\sqrt{n}$.
\end{example}

\section{Conclusion}\label{section:Conclusion}

The results in this paper illustrate how, using the geometry and topology of halfspace depth and $\mathbb{R}^d$, one can achieve better convergence bounds for the sample version of halfspace depth, as compared to general \gc\ bounds.  With improved bounds in this paper, we have improved estimates on the quality of the halfspace median statistic (see \cite{LiuSingh:Survey:1999} and \cite{Aloupis:2006}), which is applicable in statistics.  The approach in this paper is related to the ideas of the projection pursuit in \cite{Diaconis:1984} and \cite{Meckes:2009}; it is possible that incorporating such techniques may further improve the convergence rates of halfspace depth; we leave such improvements as future work.

The authors would like to thank their colleagues at Clemson University, in particular, Billy Bridges, Brian Fralix, Peter Kiessler, and June Luo for their constructive feedback on earlier versions of this work.

\bibliographystyle{plain}
\bibliography{Convergence}

\end{document}